\documentclass[11pt]{article}
\usepackage[T1]{fontenc}
\usepackage{lmodern}
\usepackage[margin=1in]{geometry}
\usepackage{amsmath,amssymb,amsthm,mathtools,microtype}
\usepackage{needspace}
\usepackage{xcolor}
\usepackage[colorlinks=true,linkcolor=blue!45!black,citecolor=blue!45!black,urlcolor=blue!45!black]{hyperref}
\hypersetup{pdftitle={On Kolmogorov's rearrangement problem and Garsia's conjecture},pdfauthor={Mark Lewko}}
\newtheorem{theorem}{Theorem}
\newtheorem{lemma}[theorem]{Lemma}
\numberwithin{equation}{section}
\newcommand{\T}{\mathbb T}
\newcommand{\N}{\mathbb N}
\newcommand{\Z}{\mathbb Z}
\newcommand{\R}{\mathbb R}
\newcommand{\F}{\mathbb F}
\newcommand{\E}{\mathbb E}
\newcommand{\PP}{\mathbb P}
\newcommand{\one}{\mathbf 1}
\newcommand{\calF}{\mathcal F}
\newcommand{\calA}{\mathcal A}
\newcommand{\calB}{\mathcal B}
\allowdisplaybreaks[2]
\title{On Kolmogorov's rearrangement problem\\ and Garsia's conjecture}
\author{Mark Lewko}
\date{}
\begin{document}
\maketitle

\begin{abstract}
We give negative answers to Kolmogorov's rearrangement problem and Garsia's conjecture. We construct a complete uniformly bounded orthonormal system for which every rearrangement admits a square-summable series divergent almost everywhere. The finite construction uses two copies of the trigonometric system in different orderings. The main ingredient is a combinatorial lemma which guarantees a prescribed ordering along an arithmetic progression in at least one of two related permutations. Its proof uses Szemer\'edi's theorem and a counting argument. A Walsh variant gives $N$-term $\{\pm1\}$-valued systems with $L^2$ maximal norm at least $c\log\log N$ in every ordering, matching Bourgain's upper bound.
\end{abstract}

\section{Introduction}

Let $\Phi=\{\phi_n\}_{n=1}^N$ denote an orthonormal system (ONS) on a probability space $(\Omega,\mu)$. For $f=\sum_{n=1}^N a_n\phi_n$, we consider the maximal partial sum operator
\begin{equation}\label{eq:maximal}
 \mathcal M f(x):=\max_{1\leq q\leq N}
 \left|\sum_{n=1}^q a_n\phi_n(x)\right|.
\end{equation}
Rademacher in 1922 \cite{Rademacher} and Menshov in 1923 \cite{Menshov} proved that
\begin{equation}\label{eq:rm}
 \|\mathcal M f\|_2\lesssim\log N\left(\sum_{n=1}^N|a_n|^2\right)^{1/2}
\end{equation}
for every ONS. In the same 1923 paper, Menshov constructed examples showing that the $\log N$ factor is sharp.

The study of rearranged trigonometric series goes back to Kolmogorov, who announced in 1926 that there is an $L^2$ function whose Fourier series can be rearranged to diverge almost everywhere \cite[p.~58]{Ulyanov1983}. The result was stated without proof in his 1927 paper with Menshov \cite[Theorem~III]{KolmogorovMenshov}. Zahorski gave a proof sketch in 1960 \cite{Zahorski}, followed by a detailed proof by Ul'yanov in 1961 \cite[Sections~4--7]{Ulyanov1961}. Also in 1961, Olevskii proved that every complete ONS admits a rearrangement for which some square-summable series diverges almost everywhere \cite[Theorem~1]{Olevskii1961}. Ul'yanov obtained a slightly stronger form of this result \cite[Section~9]{Ulyanov1961}.

An infinite ONS is called a \emph{system of convergence} if every square-summable series converges almost everywhere in its given ordering. In 1964, Olevskii constructed a complete uniformly bounded system of convergence \cite{Olevskii1964}. Carleson's celebrated 1966 theorem shows that the trigonometric system in its natural ordering is a system of convergence \cite{Carleson}. The problem of whether every ONS can be rearranged into a system of convergence has long been attributed to Kolmogorov. The earliest published statement of this problem that we have found is in Ul'yanov's 1961 survey \cite[Section~11, Problem~1]{Ulyanov1961}. The question is also explicitly stated in the 1972 article of Aleksandrov, Kolmogorov and Ul'yanov \cite[p.~166]{AKU1972}.

For a permutation $\tau$ of $[N]:=\{1,\ldots,N\}$, let $\mathcal M_\tau$ denote the maximal operator for the reordered system $\{\phi_{\tau(n)}\}_{n=1}^N$. Garsia proved in 1964 \cite{Garsia1964,Garsia1970} that, for any given coefficients $a_1,\ldots,a_N$, one can choose $\tau$ so that
\begin{equation}\label{eq:garsia}
 \|\mathcal M_\tau f\|_2\lesssim\left(\sum_{n=1}^N|a_n|^2\right)^{1/2}.
\end{equation}
The permutation may depend on the coefficients. Garsia conjectured in his 1970 book that one can choose a single permutation for which \eqref{eq:garsia} holds for all coefficients, with an absolute implicit constant \cite[\S3.6, p.~79, (3.6.3)]{Garsia1970}. He also explained how this conjecture would imply an affirmative answer to Kolmogorov's rearrangement problem \cite[pp.~79--81]{Garsia1970}. In 1989, Bourgain \cite{Bourgain} proved that, for systems satisfying $\|\phi_n\|_\infty\leq A$, one can choose $\tau$ so that
\begin{equation}\label{eq:bourgain}
 \|\mathcal M_\tau f\|_2\lesssim_A\log\log N\,\left(\sum_{n=1}^N|a_n|^2\right)^{1/2}
\end{equation}
for all coefficients. His proof uses random rearrangements and probabilistic estimates related to his work on the $\Lambda(p)$ problem \cite{BourgainLambda}. For related variational estimates, see \cite{LLVariation,LLRearranged}. Further background on orthogonal series can be found in Ul'yanov's surveys \cite{Ulyanov1961,Ulyanov1964} and Olevskii's monograph \cite{OlevskiiBook}. We prove the following.

\begin{theorem}\label{thm:main}
\begin{enumerate}
\item There is an absolute constant $c>0$ such that, for every $H>0$, there is a finite ONS $\Phi=(\phi_n)_{n=1}^N$ on a probability space $(\Omega,\mu)$, with $|\phi_n|=1$, such that every permutation $\tau$ admits a sum $f=\sum_{n=1}^N a_n\phi_n$ with real coefficients and $\sum_{n=1}^N|a_n|^2=1$ satisfying
\[
 \mu\{\mathcal M_\tau f>H\}\geq c.
\]
\item There is a complete ONS $\{\psi_n\}_{n\geq1}$ on a probability space $(\Omega,\mu)$, with $|\psi_n|=1$, such that for every permutation $\tau$ of $\N$ there are real coefficients $(a_n)\in\ell^2$ for which
\[
 \sup_{q\geq1}\left|\sum_{j=1}^q a_{\tau(j)}\psi_{\tau(j)}(x)\right|=\infty
\]
for almost every $x\in\Omega$.
\end{enumerate}
\end{theorem}

The finite construction uses two copies of the trigonometric system. Let $\T:=\R/\Z$ have normalized Lebesgue measure and write $e(t):=e^{2\pi it}$. Given a permutation $\sigma$ of $[N]$, define
\begin{equation}\label{eq:construction}
 \phi_n(x,0):=e(nx),\qquad
 \phi_n(x,1):=e(\sigma(n)x),\qquad n\in[N],
\end{equation}
on $\T\times\{0,1\}$, with each copy of $\T$ having measure $1/2$. These functions form an ONS. A rearrangement by $\tau$ gives the frequency orders $\tau$ and $\sigma\circ\tau$ on the two copies. We will choose $\sigma$ so that, for every $\tau$, one of these two orders contains an arithmetic progression in a prescribed order. A finite rearranged Fourier series with large maximal sum can then be placed on this progression.

If every permutation of $\N$ contained arbitrarily long increasing arithmetic progressions, an affirmative answer to Kolmogorov's problem would imply Carleson's theorem, since translation and dilation of the frequencies preserve the distribution of absolute partial sums. For finite intervals, Odda showed in 1975 that even increasing and decreasing three-term progressions can both be avoided \cite{Odda1975}. For permutations of $\N$, Davis, Entringer, Graham and Simmons proved in 1977 that an increasing three-term progression is unavoidable, but constructed a permutation containing no increasing or decreasing five-term progression \cite{DEGS}. Ho improved the avoidance result to four terms in 2026 \cite{Ho}.

Thus the one-copy argument suggested by arithmetic progressions does not work for the trigonometric system. For Walsh frequencies, however, the analogous containment statement holds. Bourgain \cite[Lemma~2.31]{Bourgain} proved that every ordering of $\F_2^d$ contains an affine copy of $\F_2^s$ in its natural binary order, provided $d$ is sufficiently large relative to $s$. He used this to show that an affirmative answer to Kolmogorov's problem would imply Billard's theorem \cite{Billard} on almost everywhere convergence of square-summable Walsh series.

For the counterexample we need to force a prescribed finite bad order rather than the natural order. Our trigonometric construction achieves this by using two related permutations. We show that any prescribed finite order can be found along an arithmetic progression in at least one of them. We prove this by dividing the two orderings into consecutive blocks and counting the resulting colorings. The count uses Szemer\'edi's theorem on arithmetic progressions in dense sets and is proved in Section~\ref{sec:combinatorial}. The application to the trigonometric system is given in Section~\ref{sec:trig}.

To obtain an optimal quantitative bound, we use the Walsh system. The same counting argument applies, with Szemer\'edi's cube lemma in place of his theorem on arithmetic progressions. Combined with a logarithmic lower bound for rearranged Walsh series, this gives the following result.

\begin{theorem}\label{thm:quantitative}
There are absolute constants $c,c_0>0$ such that, for every sufficiently large $N$, there is an ONS $\Phi_N=(\phi_n)_{n=1}^N$, with $\phi_n(x)\in\{-1,1\}$, having the following property. For every permutation $\tau$ of $[N]$, there is a sum $f=\sum_{n=1}^N a_n\phi_n$ with real coefficients and $\sum_{n=1}^N|a_n|^2=1$ satisfying
\begin{equation}\label{eq:sharp-lower}
 \|\mathcal M_\tau f\|_2\geq c\log\log N
\end{equation}
and
\begin{equation}\label{eq:weak-lower}
 \mu\{\mathcal M_\tau f>c\sqrt{\log\log N}\}\geq c_0.
\end{equation}
When $N$ is a power of two, the system consists of two differently labeled copies of the $N$ Walsh characters.
\end{theorem}

The Walsh construction also gives both assertions of Theorem~\ref{thm:main} with functions taking only the values $1$ and $-1$. The quantitative argument is given in Section~\ref{sec:walsh}, and we pass to complete infinite systems in Section~\ref{sec:infinite}.

Throughout the paper, $X\lesssim Y$ means that $X\leq CY$ for an absolute constant $C$, and subscripts indicate allowed dependence of this constant. Constants denoted by $c$ and $C$ may change from line to line. All logarithms are natural, and logarithmic estimates are understood for large~$N$.

\section{Permutations and arithmetic progressions}\label{sec:combinatorial}

We first prove the two-permutation statement used in the trigonometric construction.

\begin{lemma}\label{lem:two}
Let $\pi$ be a permutation of $[m]$. There exist $N\geq m$ and a permutation $\sigma$ of $[N]$ such that, for every permutation $\tau$ of $[N]$, there is an $m$-term arithmetic progression $x_1<\cdots<x_m$ in $[N]$ for which $x_{\pi(1)},\ldots,x_{\pi(m)}$ is a subsequence of at least one of
\[
 \tau(1),\ldots,\tau(N)
 \qquad\text{and}\qquad
 \sigma(\tau(1)),\ldots,\sigma(\tau(N)).
\]
\end{lemma}

We divide an ordering into $K$ consecutive blocks and color each integer by the block in which it occurs. A progression whose colors increase in the prescribed order also occurs in that order in the permutation. We therefore count the colorings which contain no such progression.

\Needspace{10\baselineskip}
\begin{lemma}\label{lem:count}
Fix integers $m\geq2$ and $K\geq m$, and a permutation $\pi$ of $[m]$. The number of colorings $c:[N]\to[K]$ containing no $m$-term arithmetic progression $x_1<\cdots<x_m$ with
\begin{equation}\label{eq:pattern}
 c(x_{\pi(1)})<\cdots<c(x_{\pi(m)})
\end{equation}
is at most
\begin{equation}\label{eq:count}
 (m-1)^N\exp(o(N)).
\end{equation}
Here $N\to\infty$ with $m$, $K$, and $\pi$ fixed.
\end{lemma}

The factor $(m-1)^N$ is best possible, since every coloring using a fixed set of $m-1$ colors avoids \eqref{eq:pattern}.

\begin{proof}
Let $\calF_N$ denote the family in the statement. We count its members by deleting the points $N,N-1,\ldots,1$ in turn. Consider a family $\calA$ of colorings of $[r]$, and let $\calA_i$ be the colorings of $[r-1]$ which extend to a member of $\calA$ by giving $r$ the color $i$. If each remaining coloring had at most $m-1$ extensions, deletion would cost at most a factor $m-1$. In general, we have the following inequality, where the intersections account for the additional extensions.
\begin{equation}\label{eq:recursion}
 |\calA|\leq(m-1)\left|\bigcup_{i=1}^K\calA_i\right|
 +\sum_{\substack{Q\subseteq[K]\\|Q|=m}}
 \left|\bigcap_{i\in Q}\calA_i\right|.
\end{equation}
To see this, consider a coloring of $[r-1]$ with $t\geq1$ extensions. It contributes $t$ to the left-hand side and $m-1+\binom tm$ to the right-hand side, and $t\leq m-1+\binom tm$.

Apply \eqref{eq:recursion} to $\calF_N$, then to each family on the right, and continue until all points have been deleted. For one nonzero term in the resulting sum, let $S$ be the points at which an intersection was chosen and let $Q_x$ be its set of $m$ colors at $x\in S$. Each of the other deletions contributes a factor $m-1$, so the term has weight $(m-1)^{N-|S|}$. It remains to bound the number of points in $S$.

We claim that any choice of a color from $Q_x$ at each point $x\in S$ extends to a single coloring in $\calF_N$. To see this, restore the points in reverse order. At a point of $S$, membership in the intersection allows every color in $Q_x$. At any other point, membership in the union allows at least one color. We can therefore make each of the chosen assignments as that point is restored. The colors outside $S$ need not be the same for different assignments on $S$.

Fix a set $Q=\{q_1<\cdots<q_m\}\subseteq[K]$. We next observe that the set of $x\in S$ with $Q_x=Q$ contains no $m$-term arithmetic progression. Otherwise, given such a progression $x_1<\cdots<x_m$, we could assign color $q_j$ to $x_{\pi(j)}$, make arbitrary allowed choices at the other points of $S$, and extend to a member of $\calF_N$. This contradicts the defining property of $\calF_N$.

Let $r_m(N)$ denote the maximal size of a subset of $[N]$ containing no $m$-term arithmetic progression, and let $R:=\binom Km$. Each of the $R$ possible sets $Q$ occurs at no more than $r_m(N)$ points of $S$, so $|S|\leq Rr_m(N)$. Summing over the choices of $S$ and the sets $Q_x$, we obtain
\begin{equation}\label{eq:explicit-count}
 |\calF_N|\leq
 \sum_{j=0}^{\min(N,Rr_m(N))}
 \binom Nj R^j(m-1)^{N-j}.
\end{equation}
By Szemer\'edi's 1975 theorem \cite{Szemeredi}, $r_m(N)=o(N)$. Since $m$ and $K$ are fixed, \eqref{eq:explicit-count} is at most $(m-1)^N\exp(o(N))$, which completes the proof.
\end{proof}

The same argument applies to a family of ordered $m$-tuples of distinct points in an $N$-element set. If every subset containing none of the underlying $m$-point sets has size at most $r$, then the number of colorings for which none of the ordered tuples has strictly increasing colors is at most
\begin{equation}\label{eq:general-count}
 \sum_{0\leq j\leq\min(N,Rr)}\binom Nj R^j(m-1)^{N-j},
 \qquad R=\binom Km.
\end{equation}
Indeed, the only property of arithmetic progressions used in the counting argument was the bound for the size of the set on which $Q_x$ is fixed. The bound \eqref{eq:general-count} also follows from the generalized Sauer--Shelah lemma \cite[Theorem~1.5]{DSS}.

\begin{proof}[Proof of Lemma~\ref{lem:two}]
We may assume $m\geq2$, since the case $m=1$ is immediate. Let $K=m^2$ and take $N$ divisible by $K$. We call a coloring \emph{balanced} if each color occurs $N/K$ times, and let $\calB_N$ denote the family of balanced colorings. Stirling's formula gives
\[
 |\calB_N|=\frac{N!}{((N/K)!)^K}=K^N\exp(o(N)).
\]
Choose $\sigma$ uniformly at random. For a fixed $c\in\calB_N$, the coloring $c\circ\sigma^{-1}$ is uniformly distributed on $\calB_N$. Writing $\calF_N$ for the family in Lemma~\ref{lem:count}, the expected number of balanced colorings such that both $c$ and $c\circ\sigma^{-1}$ belong to $\calF_N$ is
\[
 \frac{|\calF_N\cap\calB_N|^2}{|\calB_N|}
 \leq\left(\frac{(m-1)^2}{m^2}\right)^N\exp(o(N)).
\]
For $N$ sufficiently large this is less than one. Thus there is a choice of $\sigma$ for which no balanced coloring has this property.

Given a permutation $\tau$, define the balanced coloring $c$ by
\[
 c(\tau(j))=\ell
 \quad\text{when}\quad
 (\ell-1)N/K<j\leq\ell N/K.
\]
The corresponding coloring for $\sigma\circ\tau$ is $c\circ\sigma^{-1}$. By the choice of $\sigma$, at least one lies outside $\calF_N$ and hence contains a progression satisfying \eqref{eq:pattern}. Its terms occur in the required order.
\end{proof}

\section{The trigonometric construction}\label{sec:trig}

For the first $m$ exponentials, Nakata \cite{Nakata1971} obtained in 1971 a rearrangement with $L^2$ maximal norm at least $c(\log m)^{1/4}$. The optimal lower bound $c\log m$ was obtained by Karagulyan in 2020 \cite[Theorem~1.1]{Karagulyan}. For the trigonometric construction we only need the following finite consequence of Nakata's construction.

\begin{lemma}\label{lem:fourier}
There is $\delta>0$ such that, for every $H>0$, there are $m$, a permutation $\pi$ of $[m]$, and real coefficients $b_1,\ldots,b_m$ satisfying
\begin{equation}\label{eq:fourier}
 \sum_{j=1}^m b_j^2=1,\qquad
 \left|\left\{x\in\T:\max_{q\leq m}
 \left|\sum_{j\leq q}b_j e(\pi(j)x)\right|>H\right\}\right|\geq\delta.
\end{equation}
\end{lemma}

\begin{proof}
Nakata's Lemma~4 \cite{Nakata1971} gives, for arbitrarily large integers $k$, trigonometric polynomials $S_1,\ldots,S_L$ with pairwise disjoint Fourier supports such that
\[
 \int_{\T}\sum_{j=1}^L |S_j(x)|^2\,dx\lesssim k^{-1},
\]
and, on an interval of fixed positive measure, for every $x$ there are $1\leq u\leq v\leq L$ for which
\[
 \left|\sum_{j=u}^v S_j(x)\right|\gtrsim1.
\]
Order the Fourier characters blockwise according to $S_1,\ldots,S_L$. The last display is the difference of the partial sums ending after blocks $v$ and $u-1$, so one of these two partial sums has absolute value bounded below by an absolute constant. By Parseval and the disjointness of the supports, the squared sum of all coefficients is $O(k^{-1})$. After normalization the maximal sum is therefore $\gtrsim\sqrt k$ on a set of fixed positive measure.

The coefficients in the exponential expansion may be complex. Write the normalized coefficient vector as $c_j=\alpha_j+i\beta_j$, with $\alpha_j,\beta_j$ real. On the set where the corresponding maximal sum is greater than $2H$, at least one of the maximal sums associated to $(\alpha_j)$ and $(\beta_j)$ is greater than $H$. One of these two sets has at least half the measure. The selected real coefficient vector has squared sum at most one, and normalization can only increase its maximal sum. Taking $k$ sufficiently large, translating the frequencies to positive integers, and inserting missing frequencies with coefficient zero gives the lemma.
\end{proof}

For $a\in\Z$ and $d>0$,
\begin{equation}\label{eq:affine}
 \sum_{j\leq q}b_j e((a+d\pi(j))x)
 =e(ax)\sum_{j\leq q}b_j e(\pi(j)dx),
\end{equation}
so affine changes of frequency preserve the distribution of the maximal sum.

We now apply Lemma~\ref{lem:two} to the permutation $\pi$ given by Lemma~\ref{lem:fourier}, and use the resulting $\sigma$ to define the system \eqref{eq:construction}. For any rearrangement $\tau$, one of the two frequency orders contains a progression $x_j=a+dj$ in the order $x_{\pi(1)},\ldots,x_{\pi(m)}$. Give the function with frequency $x_{\pi(j)}$ on this copy of $\T$ the coefficient $b_j$, and set all other coefficients equal to zero. The nonzero terms then have the partial sums in \eqref{eq:affine}. Since the copy has measure $1/2$, the maximal sum exceeds $H$ on a set of measure at least $\delta/2$. This proves the finite assertion of Theorem~\ref{thm:main} for the trigonometric construction.

One can also obtain a real-valued system with uniform bound $\sqrt2$. Let $t\in\T$ be an additional independent variable and define
\begin{equation}\label{eq:realification}
 f_n(x,\epsilon,t):=\sqrt2\operatorname{Re}(e(t)\phi_n(x,\epsilon)).
\end{equation}
Averaging over $t$, we see that the $f_n$ form an ONS. Moreover, if $|S|>H$, then $\sqrt2|\operatorname{Re}(e(t)S)|>H$ for a set of $t$ of measure at least $1/2$. Since the coefficients in the preceding construction are real, they give a maximal sum for the $f_n$ greater than $H$ on a set of measure at least $\delta/4$.

\section{The Walsh construction}\label{sec:walsh}

Ul'yanov \cite[Section~8, Theorem~8.2]{Ulyanov1961} proved in 1961 that the Walsh system admits an almost everywhere divergent rearranged square-summable series. Billard \cite{Billard} proved in 1967 that every square-summable Walsh series converges almost everywhere in the natural ordering. Bochkarev \cite[Theorem~1]{Bochkarev1978} obtained in 1978 a rearrangement of the first $m$ Walsh functions with maximal function at least $c\sqrt{\log m}$ on a set of fixed positive measure. We prove below a lower bound of order $\log m$ for the $L^2$ maximal norm. By the Rademacher--Menshov inequality, this is optimal.

Let $\chi_u(x):=(-1)^{u\cdot x}$ denote the Walsh characters of $\F_2^d$, with uniform probability measure. Given a permutation $\sigma$ of $X:=\F_2^d$, define
\begin{equation}\label{eq:walsh-system}
 \phi_u(x,0):=\chi_u(x),\qquad
 \phi_u(x,1):=\chi_{\sigma(u)}(x),\qquad u\in X,
\end{equation}
on $X\times\{0,1\}$, giving each copy of $X$ measure $1/2$. These $N=2^d$ functions form a sign-valued ONS.

For an affine injection $L(p)=b+Tp$ from $\F_2^s$ into $\F_2^d$, we have
\begin{equation}\label{eq:walsh-affine}
 \chi_{L(p)}(x)=\chi_b(x)\chi_p(T^{\mathsf T}x).
\end{equation}
Since $T^{\mathsf T}$ sends uniform measure to uniform measure, this preserves the distribution of absolute partial sums.

We will obtain the affine analogue of Lemma~\ref{lem:two} with $\log N\geq Cm^2\log m$, where $m=2^s$ is the size of the prescribed order. At the sizes used below, $\log m\asymp\log\log N$, so it suffices to construct a rearranged Walsh system with maximal norm at least $c\log m$.

\Needspace{13\baselineskip}
\subsection{A logarithmic lower bound for a rearranged Walsh system}

\begin{lemma}\label{lem:walsh-example}
For each integer $r\geq1$, there is an ordering $p_1,\ldots,p_m$ of $\F_2^{6r}$, where $m=2^{6r}$, and real coefficients $b_1,\ldots,b_m$ with $\sum_j b_j^2=1$ such that
\begin{equation}\label{eq:walsh-seed}
 \left\|\max_{0\leq q\leq m}
 \left|\sum_{j=1}^q b_j\chi_{p_j}\right|\right\|_2
 \geq\frac{r}{\sqrt{48}}
\end{equation}
and
\begin{equation}\label{eq:walsh-seed-weak}
 \mu\left\{\max_{0\leq q\leq m}
 \left|\sum_{j=1}^q b_j\chi_{p_j}\right|>\frac{\sqrt r}{8}\right\}
 \geq\frac1{12}.
\end{equation}
\end{lemma}

\begin{proof}
We construct polynomials on $(\Z/4\Z)^{2r}$ with disjoint Fourier supports, then express them as Walsh polynomials with disjoint supports. Their partial sums can then be realized as partial sums of a rearranged Walsh series by placing the characters of each polynomial consecutively.

Let $x=(x_0,\ldots,x_{r-1})$ and $y=(y_0,\ldots,y_{r-1})$ denote independent uniformly distributed elements of $(\Z/4\Z)^r$, and define
\[
 \epsilon(0)=\epsilon(2)=0,\qquad
 \epsilon(1)=1,\qquad\epsilon(3)=-1.
\]
For $1\leq n\leq r$ and a word $a=(a_0,\ldots,a_{n-1})\in\{0,1,2,3\}^n$, let
\begin{equation}\label{eq:triangular}
 z_j(a):=y_j+\sum_{i=0}^j a_i x_{j-i}\pmod4,
 \qquad 0\leq j<n,
\end{equation}
and
\begin{equation}\label{eq:quaternary-functions}
 F_{n,a}(x,y):=\epsilon(z_{n-1}(a))
 \prod_{j=0}^{n-2}\one_{\{z_j(a)=0\}}.
\end{equation}
For $n=1$ the product is omitted. For fixed $x$ and $a$, the variables $z_j(a)$ are independent and uniformly distributed as $y$ varies. It follows that
\begin{equation}\label{eq:block-energy}
 \|F_{n,a}\|_2^2=2\cdot4^{-n},\qquad
 \sum_{n=1}^r\sum_{a\in\{0,1,2,3\}^n}\|F_{n,a}\|_2^2=2r.
\end{equation}

To verify disjointness of the Fourier supports, expand the indicators in \eqref{eq:quaternary-functions} and use $\epsilon(z)=(i^z-i^{-z})/(2i)$. Each character which occurs has exponent
\[
 \sum_{j=0}^{n-1}u_jz_j(a),\qquad
 u_j\in\Z/4\Z,\quad u_{n-1}\in\{1,3\}.
\]
The coefficients of the $y$-coordinates determine $n$ and the $u_j$. Since $u_{n-1}$ is odd, the coefficient $u_{n-1}a_0$ of $x_{n-1}$ then determines $a_0$. Comparing the remaining $x$-coefficients in reverse order determines $a_1,a_2,\ldots$ in turn. Thus no character occurs in two different functions $F_{n,a}$. By Parseval and \eqref{eq:block-energy}, the squared coefficient norm of their sum is $2r$.

We order the functions lexicographically in $a$, putting a word before its proper extensions. Suppose first that $x_0$ is odd. The equations $z_j(a)=0$ can then be solved successively, giving a unique word $w=(w_0,\ldots,w_{r-1})$ of length $r$. We take the partial sum ending just before $w$. At length $n$, a nonzero term must have first $n-1$ digits $w_0,\ldots,w_{n-2}$, and it precedes $w$ precisely when its last digit is less than $w_{n-1}$. The term with last digit $w_{n-1}$ is zero. The contribution from words of length $n$ is therefore
\[
 \sum_{b<w_{n-1}}\epsilon(x_0(b-w_{n-1}))
 =-\epsilon(x_0)\one_{\{w_{n-1}\in\{1,2\}\}}.
\]
Writing
\[
 D_r(w):=|\{j:w_j\in\{1,2\}\}|,
\]
this partial sum has absolute value $D_r(w)$, since all the contributions have the same sign. For fixed $x$ with $x_0$ odd, the word $w$ is uniformly distributed as $y$ varies. In particular,
\begin{equation}\label{eq:digit-moment}
 \E D_r=r/2,\qquad \E D_r^2\geq r^2/4.
\end{equation}

This gives a lower bound of order $\sqrt r$ after normalization by \eqref{eq:block-energy}. We obtain a larger $L^2$ lower bound by also considering points where some initial coordinates of $x$ and $y$ vanish. On these sets the same construction occurs in fewer coordinates, with each term repeated many times.

For $0\leq k<r$, we define
\[
 E_k:=\{x_0=\cdots=x_{k-1}=y_0=\cdots=y_{k-1}=0,
                  \ x_k\text{ odd}\}.
\]
The sets $E_k$ are disjoint and have measures $16^{-k}/2$. On $E_k$, let $x'=(x_k,\ldots,x_{r-1})$ and $y'=(y_k,\ldots,y_{r-1})$. The functions with $n\leq k$ vanish, and for $1\leq\ell\leq r-k$ the definition gives
\[
 F_{k+\ell,a}(x,y)
 =F_{\ell,(a_0,\ldots,a_{\ell-1})}(x',y').
\]
The right-hand side does not depend on the last $k$ digits of $a$. Consequently, each function in the construction on $(x',y')$ occurs $4^k$ times.

Let $w$ denote the solution word for $(x',y')$, and consider the partial sum ending just before the length-$r$ word obtained by appending $k$ zeros to $w$. If a word in the shorter construction is not a prefix of $w$, then its first differing digit determines whether it precedes the cutoff, independently of the $k$ appended digits. If the word is a prefix of $w$, its function is zero. Thus this partial sum is exactly $4^k$ times the partial sum previously considered for $(x',y')$.

Let $S_r$ denote this partial sum on each $E_k$, and set $S_r=0$ outside their union. The factor $16^k$ from squaring cancels the measure factor $16^{-k}$, so \eqref{eq:digit-moment} gives
\begin{equation}\label{eq:sum-events}
 \|S_r\|_2^2
 \geq\sum_{k=0}^{r-1}\frac{16^{-k}}2\,16^k\,
                 \frac{(r-k)^2}{4}
 =\frac18\sum_{j=1}^r j^2
 \geq\frac{r^3}{24}.
\end{equation}

It remains to express the construction in Walsh characters. To preserve disjointness of Fourier supports, we use three bits for each coordinate modulo four. Let $\alpha,\beta,\gamma$ denote independent uniform bits and define
\[
 z:=\alpha+2\beta+\gamma\pmod4,\qquad
 A:=(-1)^\alpha,\quad B:=(-1)^\beta,\quad D:=(-1)^\gamma.
\]
The variable $z$ is uniform modulo four and independent of $\gamma$, and a direct computation gives
\[
\begin{array}{c|c|c}
 h&i^\gamma i^{hz}&\text{Walsh support}\\ \hline
 0&i^\gamma&\{1,D\}\\
 1&BD\,i^\alpha&\{BD,ABD\}\\
 2&A\,i^{-\gamma}&\{A,AD\}\\
 3&B\,i^{-\alpha}&\{B,AB\}.
\end{array}
\]
The supports follow by expanding factors such as $i^\alpha=(1+i)/2+(1-i)A/2$, and are pairwise disjoint.

Make this substitution independently in each of the $2r$ coordinates of $(x,y)$, writing $\gamma_1,\ldots,\gamma_{2r}$ for the additional bits, and set $\rho:=i^{\gamma_1+\cdots+\gamma_{2r}}$. Multiplication by $\rho$ factors coordinatewise, so the table shows that distinct characters on $(\Z/4\Z)^{2r}$ give disjoint Walsh supports. It follows that the $\rho F_{n,a}$ have disjoint Walsh supports on $6r$ bits. Taking real parts gives
\[
 G_{n,a}:=(\operatorname{Re}\rho)F_{n,a}.
\]
Since Walsh characters are real-valued, taking real parts does not introduce any new frequencies, so the supports remain disjoint. Moreover, $\operatorname{Re}\rho$ is independent of $(x,y)$ and has absolute value one with probability $1/2$, and zero otherwise. Combining this with \eqref{eq:block-energy} and \eqref{eq:sum-events} gives
\[
 \sum_{n,a}\|G_{n,a}\|_2^2=r,\qquad
 \|(\operatorname{Re}\rho)S_r\|_2^2\geq r^3/48.
\]

We expand each $G_{n,a}$ in Walsh characters, place its characters consecutively, and order these blocks according to the ordering of the $F_{n,a}$. Any unused characters are included with coefficient zero. At the ends of the blocks, the partial sums agree with those of the $G_{n,a}$, while the squared coefficient norm is $r$. Dividing the coefficients by $\sqrt r$ proves \eqref{eq:walsh-seed}.

On $E_0$ we have $|S_r|=D_r$. Since $\E D_r=r/2$ and $D_r\leq r$, it follows that $\PP(D_r\geq r/4)\geq1/3$. Together with $\mu(E_0)=1/2$, this gives $|S_r|\geq r/4$ on a set of measure at least $1/6$. Independently, $|\operatorname{Re}\rho|=1$ with probability $1/2$. Thus the normalized maximal sum is at least $\sqrt r/4$ on a set of measure at least $1/12$, proving \eqref{eq:walsh-seed-weak}.
\end{proof}

\subsection{Affine subspaces and two permutations}

We use the following finite field form of Szemer\'edi's 1969 cube lemma \cite[p.~93]{SzemerediCube}.

\begin{lemma}[Szemer\'edi's cube lemma, finite field form]\label{lem:cube}
Let $m=2^s$, where $s\geq1$, and let $N=2^d$. A subset of $\F_2^d$ containing no affine $s$-dimensional subspace has size less than $2N^{1-1/m}$.
\end{lemma}

\begin{proof}
Let $A\subseteq\F_2^d$ have density $\delta$, and let $x,v_1,\ldots,v_s$ be independent uniformly distributed elements of $\F_2^d$. We write $C_s(A)$ for the probability that
\[
 x+\omega_1v_1+\cdots+\omega_sv_s\in A
 \quad\text{for every }\omega\in\F_2^s.
\]
Fixing $v_s$, this is the condition that an $(s-1)$-dimensional cube be contained in $A\cap(A+v_s)$. Starting with $C_0(A)=\delta$, induction and Jensen's inequality give
\[
 C_s(A)=\E_v C_{s-1}(A\cap(A+v))
 \geq\E_v\left(\frac{|A\cap(A+v)|}{N}\right)^{2^{s-1}}
 \geq\delta^{2^s}.
\]
Here we used the identity $\E_v|A\cap(A+v)|/N=\delta^2$.

If the directions are linearly independent, the cube is an affine $s$-dimensional subspace. Under the hypothesis on $A$, every cube counted by $C_s(A)$ must therefore have dependent directions. On the other hand, each nonzero binary linear combination of the directions vanishes with probability $1/N$, so the probability of dependence is at most $(2^s-1)/N$. We conclude that $\delta^m<m/N$, and hence
\[
 |A|<m^{1/m}N^{1-1/m}\leq2N^{1-1/m}.\qedhere
\]
\end{proof}

\begin{lemma}\label{lem:walsh-two}
There is an absolute constant $C$ such that the following holds. Let $s\geq1$, let $m=2^s$, and fix an ordering $p_1,\ldots,p_m$ of $\F_2^s$. If
\begin{equation}\label{eq:dimension}
 d\geq C m^2\log m,
\end{equation}
then there is a permutation $\sigma$ of $X=\F_2^d$ with the following property. For every ordering $\tau$ of $X$, an affine image of $p_1,\ldots,p_m$ occurs as a subsequence of $\tau$ or of $\sigma\circ\tau$.
\end{lemma}

\begin{proof}
Let $N=2^d$, $K=16m^2$, and $R=\binom Km$. We apply \eqref{eq:general-count} to the ordered affine images of $p_1,\ldots,p_m$. Since $\log R\leq m\log(16em)$, taking $C$ sufficiently large in \eqref{eq:dimension} ensures that
\[
 N^{1/m}>2R^2.
\]
It follows from Lemma~\ref{lem:cube} that every set containing no affine $s$-dimensional subspace has size less than $N/R^2$.

Let $\calF$ denote the family of colorings $c:X\to[K]$ for which there is no affine injection $L:\F_2^s\to X$ satisfying
\[
 c(L(p_1))<\cdots<c(L(p_m)).
\]
Applying \eqref{eq:general-count} with $r<N/R^2$, we have
\begin{align*}
 |\calF|
 &\leq\sum_{0\leq j\leq N/R}\binom NjR^j(m-1)^{N-j}\\
 &\leq R^{N/R}\sum_{j=0}^N\binom Nj(m-1)^{N-j}
 =R^{N/R}m^N\leq(2m)^N.
\end{align*}
The last inequality follows from $R^{1/R}\leq2$.

Increasing $C$ if necessary, \eqref{eq:dimension} gives $d\geq2s+4$, and hence $K\mid N$. The number of balanced colorings is at least
\[
 \frac{N!}{((N/K)!)^K}\geq(K/e)^N.
\]
For a uniformly random permutation $\sigma$, the expected number of balanced colorings for which both $c$ and $c\circ\sigma^{-1}$ belong to $\calF$ is therefore bounded by
\[
 \frac{(2m)^{2N}}{(K/e)^N}=(e/4)^N<1.
\]
We may therefore choose $\sigma$ such that no balanced coloring has this property. For any ordering $\tau$, color its consecutive blocks of size $N/K$ as in Lemma~\ref{lem:two}. At least one of the colorings corresponding to $\tau$ and $\sigma\circ\tau$ does not belong to $\calF$, and the resulting ordered affine image is the required subsequence.
\end{proof}

\subsection{Proof of Theorem~\ref{thm:quantitative}}

Let $p_1,\ldots,p_m$ and $b_1,\ldots,b_m$ be given by Lemma~\ref{lem:walsh-example}, so that $m=2^{6r}$, and put $s=6r$. For any $d$ satisfying \eqref{eq:dimension}, Lemma~\ref{lem:walsh-two} provides a permutation $\sigma$ with which to define \eqref{eq:walsh-system}. Every rearrangement of this system contains an ordered affine image of the Walsh example on one of the two copies. Let $f$ be the sum obtained by assigning the coefficients from Lemma~\ref{lem:walsh-example} to these functions and setting the remaining coefficients equal to zero. By \eqref{eq:walsh-affine}, the maximal sum on this copy has the same distribution as in that lemma. Recalling that the copy has measure $1/2$, we obtain
\[
 \|\mathcal M_\tau f\|_2\geq\frac{r}{\sqrt{96}},\qquad
 \mu\{\mathcal M_\tau f>\sqrt r/8\}\geq\frac1{24}.
\]

Condition \eqref{eq:dimension} holds whenever $d\geq C'(r+1)2^{12r}$ for an absolute constant $C'$. For $N=2^d$ sufficiently large, we take the largest $r$ satisfying this inequality. Then $r\asymp\log d\asymp\log\log N$, which gives both estimates in Theorem~\ref{thm:quantitative}. For general $N$, take the construction at the largest power of two not exceeding $N$ and include independent mean-zero signs on additional coordinates. These functions are orthogonal to the original system and to one another, and giving them coefficient zero preserves both lower bounds in every ordering. This proves the theorem with $c_0=1/24$, as well as the sign-valued finite assertion of Theorem~\ref{thm:main}.

\section{Complete systems and almost everywhere divergence}\label{sec:infinite}

For each $r\geq1$, choose a finite system $\Phi_r$ for which every ordering admits real coefficients with squared sum one and maximal sum greater than $4^r$ on a set of measure at least $c>0$, where $c$ is independent of $r$. Place these systems on independent probability spaces $\Omega_r$. We may assume that all the functions have mean zero. This already holds for the trigonometric construction. In the Walsh construction we multiply every function in a finite block by the same independent mean-zero sign, which does not change any absolute partial sum. The union of the finite systems is then orthonormal on $\Omega=\prod_{r\geq1}\Omega_r$.

Consider an ONS containing this union and fix an ordering of it. Let $I_r$ denote the smallest interval of positions containing the functions of $\Phi_r$. We may choose an infinite set $J\subseteq\N$ so that the intervals $I_r$, $r\in J$, are disjoint and occur in increasing order of $r$. Indeed, a finite initial segment of the ordering meets only finitely many blocks. After any chosen interval there are therefore arbitrarily large $r$ for which all functions of $\Phi_r$ occur to its right.

For $r\in J$, choose the coefficients for the induced ordering of $\Phi_r$ and multiply them by $2^{-r}$. Set all other coefficients equal to zero. The resulting coefficients are real and square-summable, with squared sum $\sum_{r\in J}4^{-r}$. For each $r\in J$, there is a set $E_r$ of measure at least $c$ on which some partial sum within $I_r$ differs from the partial sum just before $I_r$ by more than $2^r$. No other chosen block contributes inside $I_r$, so at least one of these two partial sums has absolute value greater than $2^{r-1}$. Since the events $E_r$ depend on separate coordinates, they are independent. The second Borel--Cantelli lemma now implies that the partial sums are unbounded almost everywhere.

For the trigonometric completion, extend the permutation of a finite block to a permutation $\sigma$ of $\Z$ by fixing zero and the unused positive integers and setting $\sigma(-n)=-\sigma(n)$. Let $u_n$ denote the functions in \eqref{eq:construction}, now indexed by $n\in\Z$, and define
\[
 v_n(x,\epsilon):=(-1)^\epsilon u_n(x,\epsilon).
\]
The two families together form an ONS, since
\[
 \langle u_n,u_m\rangle=\langle v_n,v_m\rangle=\delta_{nm},\qquad
 \langle u_n,v_m\rangle=\tfrac12(\delta_{nm}-\delta_{\sigma(n),\sigma(m)})=0.
\]
Completeness follows from the identities
\[
 \frac{u_n+v_n}{2}=\one_{\{\epsilon=0\}}e(nx),\qquad
 \frac{u_n-v_n}{2}=\one_{\{\epsilon=1\}}e(\sigma(n)x).
\]
Thus each coordinate space has an orthonormal basis containing the constant function and the original finite block, with every basis function of absolute value one. Taking finite tensor products, with the constant function in all but finitely many coordinates, gives a complete system on $\Omega$.

For the real-valued trigonometric construction, we can retain the uniform bound $\sqrt2$. To include the variable $t$ from \eqref{eq:realification}, take products of the coordinate basis functions with $e(kt)$, $k\in\Z$. Form finite tensor products and then replace each nonreal conjugate pair by its normalized real and imaginary parts. The identity $\sigma(-n)=-\sigma(n)$ ensures that the complex basis is invariant under conjugation. This gives a complete real ONS containing the functions in \eqref{eq:realification}. Taking real and imaginary parts after forming the tensor products keeps the uniform bound at $\sqrt2$.

For the Walsh construction, let $Y=\F_2^d\times\{0,1\}$ denote the space of a finite block and write its functions as $\phi_u$, $u\in\F_2^d$. Define $\widetilde\phi_u(x,\epsilon):=(-1)^\epsilon\phi_u(x,\epsilon)$. The preceding orthogonality computation shows that the functions $\phi_u,\widetilde\phi_u$ form a sign-valued orthonormal basis of $L^2(Y)$. This basis need not contain the constant function, since $\sigma$ need not fix zero. Let $\eta$ denote the independent sign used above to make the block mean zero. On $Y\times\{-1,1\}$, take the ordinary Walsh basis on $Y$, independent of $\eta$, together with the functions $\eta\phi_u$ and $\eta\widetilde\phi_u$. The former span the functions independent of $\eta$ and the latter span the functions odd in $\eta$. Their union is a complete sign-valued ONS containing both the constant function and the mean-zero block $\eta\Phi_r$.

Finite tensor products of these coordinate bases give a complete sign-valued system on $\Omega$ containing all the finite blocks, to which the divergence argument applies. This completes the proof of Theorem~\ref{thm:main}.

\section{Some open problems}

The Rademacher--Menshov inequality was strengthened in 2012 to a square-variation estimate with the same $\log N$ factor \cite{LLVariation}. The sharp rearrangement estimate for square variation is not known. For systems uniformly bounded by $A$, an $L^2$ operator bound of $O_{A,\varepsilon}((\log N)^{9/22+\varepsilon})$ was obtained after rearrangement in 2015 \cite{LLRearranged}. Can this be improved to $O_A(\log\log N)$?

Bourgain noted that \eqref{eq:bourgain} also holds under the weaker assumption $\|\phi_n\|_{L^p}\leq A$ for some fixed $p>2$, with the implicit constant depending on $p$ and $A$ \cite[pp.~211--212]{Bourgain}. Without any assumption beyond orthonormality, can rearrangement improve the Rademacher--Menshov bound? More precisely, does there exist a function $\eta(N)\to0$ such that every $N$-term ONS admits a permutation $\tau$ for which
\[
 \|\mathcal M_\tau f\|_2
 \lesssim\eta(N)\log N\,\left(\sum_{n=1}^N|a_n|^2\right)^{1/2}
\]
for every $f=\sum_{n=1}^N a_n\phi_n$, uniformly over all finite ONSs?

\Needspace{6\baselineskip}
\section*{Acknowledgements and AI assistance}

I am grateful for Jean Bourgain's encouragement and for our discussions of this problem. I first learned of it from his paper \cite{Bourgain} as a graduate student, and it was the first problem I asked him about when we met at the Institute for Advanced Study. He suggested that ideas from the solution of the Kadison--Singer problem \cite{MSS} might be useful. Our conversations left me with the impression that he expected a positive answer, although he never explicitly said so. I would also like to thank Gady Kozma for correspondence concerning the history of the problem. After the first version of this paper appeared, he wrote to tell me that Olevskii had considered a two-copy approach of this type in the 1990s or earlier, and that he and Olevskii had pursued it using a direct counting argument for permutations, but that route did not give the estimate needed to complete the argument.

\Needspace{8\baselineskip}
I used several ChatGPT models extensively in developing this paper, including GPT-5.5, GPT-5.6, GPT-6 Astra, and GPT-6 Pro. An initial proof I developed with this assistance was more complicated and used a different Walsh construction together with the hypergraph container results of Saxton and Thomason \cite{SaxtonThomason}. I later found the simpler counting argument presented here. I had previously studied related questions about permutations of the integers while trying to show that a positive answer to Kolmogorov's problem would imply Carleson's theorem. At the time I did not view these questions as a route to a counterexample. ChatGPT also helped obtain the logarithmic $L^2$ lower bound for rearranged Walsh series in Section~\ref{sec:walsh}, which improves the quantitative lower bounds discussed there. I used ChatGPT and Claude Opus 5.5 for copyediting and suggestions on the presentation. I take responsibility for all content of the paper.

\Needspace{8\baselineskip}
\begingroup
\small

\medskip
\noindent\textsc{Enfield, New Hampshire, USA}\\[2pt]
\textit{Email address:} \href{mailto:mlewko@gmail.com}{\texttt{mlewko@gmail.com}}
\endgroup
\end{document}